\documentclass[12pt]{article}
{\begin{Sbox}\begin{minipage}}%
{\end{minipage}\end{Sbox}\fbox{\TheSbox}}%

\usepackage[psamsfonts]{amssymb}
\usepackage{amsmath, amsthm, anysize, enumerate, color, url}
\usepackage{graphicx}
\usepackage{epstopdf}
\usepackage{epsfig}
\usepackage{subfigure}
\usepackage{breakurl}
\usepackage{algorithm}
\usepackage{algpseudocode}
\usepackage{natbib}
\usepackage{booktabs}
\usepackage{threeparttable}
\usepackage[colorlinks,citecolor=blue,urlcolor=blue]{hyperref}

\newtheorem{theorem}{Theorem} 

\newtheorem{corollary}{Corollary}

\newtheorem{lemma}{Lemma} 

\newtheorem{definition}{Definition}

\theoremstyle{definition}
\newtheorem{remark}{Remark}  

\newcommand{\E}{\mathbb{E}}
\newcommand{\R}{\mathbb{R}}
\newcommand{\N}{\mathbb{N}}

\renewcommand{\P}{\mathbb{P}}

\newcommand{\bs}{\boldsymbol}

\numberwithin{equation}{section}
\theoremstyle{plain}

\usepackage{setspace}
\begin{document}
\title{Differentially private estimation in a class of directed network models}
\author{
Lu Pan
\hspace{15mm}
Jianwei Hu
\hspace{15mm}
Peiyan Li
\\~~\\
Central China Normal University
\thanks{Department of Statistics, Central China Normal University, Wuhan, 430079, China.
\texttt{Emails:} panlulu@mails.ccnu.edu.cn.
jwhu@mail.ccnu.edu.cn.
lipeiyan@mail.ccnu.edu.cn.}
}
\date{}

\maketitle

\begin{abstract}
Although the theoretical properties in the $p_0$ model based on
a differentially private bi-degree sequence have been derived, it is still lack of a  unified theory for a general class of directed network models with the $p_{0}$ model as a special case.
We use the popular Laplace data releasing method to output the bi-degree sequence of directed networks,
which satisfies the private standard--differential privacy.
The method of moment is used to estimate unknown parameters.
We prove that the differentially private estimator is uniformly consistent and asymptotically normal under some conditions.
Our results are illustrated by the Probit model.
We carry out simulation studies to illustrate theoretical results and provide a real data analysis.

 \\

\end{abstract}

\vskip 5 pt

\section{Introduction}
As more and more network data are being made public, data privacy has received wide attention because of sensitive information (e.g., sex-partner relationships).
Using anonymous or non-anonymous nodes to publish these sensitive data may cause serious privacy problems and even lead to legal actions. For example, \cite{Jernigan:Mistree:2009} successfully predicted the sexual orientation of Facebook users by
using their friendships' public information. It has been proven that this anonymized method can expose privacy through re-identification technology (e.g., \cite{Hay:2008,Narayanan:Shmatikov:2009}).

\cite{Dwork:Mcsherry:Nissim:Smith:2006} introduced the \emph{differential privacy} for protecting privacy when releasing sensitive or confidential data.
Since then, it has been widely used as a privacy standard for applications.
It guarantees that the output distribution for two neighboring data sets does not differ too much.
The differential privacy is a good framework for privacy protection, which reduces the leakage of privacy risks and ensures the utility of the data, and has been widely used as a privacy standard when releasing network data (e.g., \cite{Hay:2009,Lu:Miklau:2014,Karwa:Slakovic:2016}).

In many cases, the degree sequence is the only information available and many other important properties are constrained by it. However, the degree may carry confidential and sensitive information, such as the sexually transmitted disease [\cite{Helleringer:Kohler:2007}]. To solve it, we can add noises into degrees. For example, \cite{Hay:2009} proposed efficient algorithms for releasing and
denoising the degree distribution of undirected networks.
\cite{Karwa:Slakovic:2016} derived the differentially private estimators of parameters in the $\beta$-model,
and proved that they are consistent and asymptotically normally distributed.
\cite{Yan:2021} developed differentially private inferences in the $p_0$ model for directed networks with a bi-degree sequence.

In this paper, we aim to establish the unified asymptotically theoretical framework in a class of directed networks for differentially private analysis.
We release bi-degree sequences by adding discrete Laplacian noises and use the method of moment to estimate the unknown parameters.
This is inspired by \cite{Yan:2021}. In a general class of directed network models, we show that
the differentially private estimator is uniformly consistent and asymptotically normal under some conditions
when the number of nodes tends to infinity.
In the case of no differential privacy,
\cite{fan:2021} established the unified theoretical framework for directed graphs with bi-degree sequence.
Here, we focus on differentially private inference, which is different from [\cite{fan:2021}].

The rest of this article is organized as follows. Section \ref{section-main} introduces the necessary background of differential privacy and presents the estimation of the noisy bi-degree sequence based on the moment equation, and obtains unified asymptotic properties for the differentially private estimation as the number of nodes goes to infinity. Section \ref{s3} applies our theoretical results in Section \ref{section-main} to the Probit model. Section \ref{section:simulation} carries out the simulations under the Probit model and a real data analysis. Some further discussion is given in Section \ref{section-discussion}. All proofs are deferred to the Appendix section.

\section{Main results}
\label{section-main}

Consider a simple directed graph $G_{n}$ with $n$ nodes denoted by ``$1,\ldots,n$" and no multiple edges or self-loops.
Let the adjacency matrix of $G_n$ be $A=(a_{i,j})_{n\times n}$,
where $a_{i,j}$ is the weight of the directed edge from head node $i$ to tail node $j$.
Let $\bs d=(d_1, \ldots, d_n)^\top$ be the out-degree sequence of $G_n$, where
$d_i=\sum_{j \neq i} a_{i,j}$ is the out-degree of node $i$.
In a similar vein, denote by the in-degree sequence $\bs b=(b_1, \ldots, b_n)^\top$, where
$b_j=\sum_{j \neq i} a_{i,j}$ is the in-degree of node $j$.
We call the pair $( \bs d, \bs b)$ or $\{(d_1, b_1), \ldots, (d_n, b_n)\}$ the bi-degree sequence.

Assume that the probability mass function of the edge weights $a_{i,j}$ ($i\neq j$) have the following general form [\cite{fan:2021}]:
 \begin{equation}\label{eee}
 a_{i,j}=a|\alpha_{i},\beta_{j}\sim f((\alpha_{i}+\beta_{j})a),
 \end{equation}
where $f(\cdot)$ is a probability mass function, $\alpha_{i}$ denotes the strength parameter of node $i$ from an outgoing edge and $\beta_{j}$ denotes the strength parameter of node $j$ from an incoming edge.
We note that the value of $f(\cdot)$ in Equation (\ref{eee}) is invariant under the transforms $(\bs{\alpha},\bs{\beta})$ to $(\bs{\alpha}-c,\bs{\beta}+c)$ for a constant $c$.
For the model identification, we set $\beta_{n}=0$ hereafter.

\subsection{Differential privacy}

Let $D=(D_{1},\ldots,D_{n})$ represent the original database. Our primary concern lies in data releasing mechanisms
that accept $D$ as input and subsequently produce a database $S=(S_1, \ldots, S_k)$ for public utilization.
In general, the size of $S$ is different from $D$.
Define a data releasing mechanism $\mathcal{A}(\cdot|D)$
that specifies a conditional probability distribution for outputting $S$ given $D$.
We call $S=(S_1, \ldots, S_k)$ a sanitized database. Schematically:\\

input database $D=(D_{1},\ldots,D_{n})\xrightarrow[sanitize]{\mathcal{A}(S|D)}$ output database $S=(S_1, \ldots, S_k)$.\\

Consider two neighboring databases $D$ and $D'$, which differ only in a single element.
Roughly speaking, differential privacy requires that
$\mathcal{A}(\cdot|D)$ and $\mathcal{A}(\cdot|D')$ are almost indistinguishable.
Formally, the data releasing mechanism $\mathcal{A}$ is \emph{$\epsilon$-differentially private}
if for all subsets of output $B\subseteq Range(\mathcal{A})$ and two neighboring databases $D$ and $D'$ [\cite{Dwork:Mcsherry:Nissim:Smith:2006}],
\[
\mathcal{A}(S\in B |D) \leq e^{\epsilon }\times \mathcal{A}(S\in B |D'),
\]
where $\epsilon$ is the privacy parameter that denotes the level of privacy.
Smaller the $\epsilon$ is,  the privacy protection is stronger.

Differential privacy in network data is separated into node differential privacy [\cite{Kasiviswanathan:Nissim:Raskhodnikova:Smith;2013}] and edge differential privacy (EDP) [\cite{Nissim:Raskhodnikova:Smith:2007}],
based on the definition of the graph neighbor.
If two graphs differ in only a single edge, we call these two graphs neighbors,
in which differential privacy is EDP.
In a similar way, two graphs are neighboring if by removing one node and all
edges incident to the node in one graph, one obtains the other graph, then differential privacy is node differential privacy.
In this paper, we focus on EDP [\cite{Hay:2009}].
Let $\delta(G,G')$ be the Hamming distance, i.e., the number of different edges between graphs $G$ and $G'$.
The EDP is defined below.

\begin{definition} {\rm (EDP).}
Let $\epsilon>0$ be a parameter for privacy level. A randomized mechanism
$\mathcal{A}(\cdot |G)$ satisfies $\epsilon$-EDP if
\[
\sup_{ G, G^\prime \in \mathcal{G}, \delta(G, G^\prime)=1 } \sup_{ S\in \mathcal{S}}  \frac{ \mathcal{A}(S|G) }{ \mathcal{A}(S|G^\prime ) } \le e^\epsilon,
\]
where $\mathcal{G}$ denotes the input space of directed graphs and
$\mathcal{S}$ denotes a set of outcomes.
\end{definition}

For a function $f: \mathcal{G} \rightarrow \mathbb{R}^{k}$, let $\Delta_{f}$ be its global sensitivity:
\[
\Delta_{f} = \max_{ \delta(G, G')=1 }\| f(G) - f(G') \|_1,
\]
where $\| \cdot \|_1$ is the $L_1$-norm.
If the outputs are network statistics,
the Laplace mechanism [\cite{Dwork:Mcsherry:Nissim:Smith:2006}] can be used.
This mechanism adds Laplace noises, where it is proportional to $\Delta_{f}$.
In cases where $f(G)$ takes integer values, discrete Laplace random variables can be used as noises, as demonstrated in [\cite{Karwa:Slakovic:2016}].

\begin{lemma}\label{lemma:DLM}{\rm(Discrete Laplace Mechanism)}.
Let $f:\mathcal{G} \to \R^k$.
The discrete Laplace mechanism works by adding noise to $f$:
\[
\mathcal{A}(\cdot |G)=f(G)+(Z_1, \ldots, Z_k),
\]
where $Z_1, \ldots, Z_k$ be $i.i.d.$ random variables from discrete Laplace distribution with  probability mass function defined as follows:
\[
\P(Z=z)= \frac{1-\lambda}{1+\lambda} \lambda^{|z|},~~z \in \mathbb{Z}, ~~\lambda\in(0,1).
\]
The discrete Laplace mechanism is $\epsilon$-edge differential privacy, where $\epsilon= -\Delta_{f}\log \lambda$.

\end{lemma}

\textbf{Post-Processing:} Differentially private mechanisms are immune to post-processing.
This means that any function with a differentially private mechanism will remain differentially private.
More formally, let $f$ be an output of a $\epsilon$-differentially private algorithm and $g$ be any function. Then
$g(f(G))$ remains $\epsilon$-differentially private.

\subsection{Differential privacy in directed random graph models}

 The discrete Laplace mechanism in Lemma \ref{lemma:DLM} is used to reveal the bi-degree sequence $(\bs d, \bs b)$, guaranteeing EDP. Since adding or removing a directed edge will increase or decrease the degrees of two corresponding nodes by one each.
 When we consider $f$ to be the bi-degree sequence, we find that the global sensitivity $\Delta_{f}$ equals 2.
So, we have $\lambda=\exp(-\epsilon/2)$. Then we obtain the output $( \bs {\tilde{d}}, \bs{ \tilde{b}})$:
\begin{equation}\label{equation:df}
\begin{array}{lcl}
\tilde{d}_i=d_i + \gamma_i,~~i=1, \ldots, n,\\
\tilde{b}_j = b_j + \zeta_j,~~j=1, \ldots, n, \\
\end{array}
\end{equation}
where the random variables $\{ \gamma_i \}_{i=1}^n$ and $\{ \zeta_j \}_{j=1}^n$ are mutually independent and distributed by the symmetric discrete Laplace distribution with the parameter $\lambda$.
So $\mathbb{E}(\gamma_{i})=\mathbb{E}(\zeta_{j})=0$.
Let $\mu(\cdot)$ denote the expectation of $f(\cdot)$ and ${\bs \theta}=(\alpha_{1},\ldots,\alpha_{n},\beta_{1},\ldots,\beta_{n-1})^{\top}$. These arguments motivate the use of the following estimating equations:
 \begin{equation}\label{eqqq1}
    \begin{split}
&\tilde{d}_i =\sum_{k=1;k\neq i}^{n}\mu(\hat{\alpha}_{i}+\hat{\beta}_{k}), ~~i=1,\ldots,n,\\
&\tilde{b}_j=\sum_{k=1;k\neq j}^{n}\mu(\hat{\alpha}_{k}+\hat{\beta_{j}}), ~~j=1,\ldots,n-1,
\end{split}
    \end{equation}
where ${\widehat{\bs\theta}}=(\hat{\alpha}_{1},\ldots,\hat{\alpha}_{n},\hat{\beta}_{1},\ldots,\hat{\beta}_{n-1})^{\top}$ and $\hat{\beta}_{n}=0$. Since $ \bs {\tilde{d}}$ and $ \bs {\tilde{b}}$ satisfy edge differential privacy, ${\widehat{\bs \theta}}$ is edge differential privacy estimator of ${\bs \theta}$ according to post-processing. In the following, we will conduct a rigorous analysis of the asymptotic properties of ${\widehat{\bs \theta}}$.
\subsection{Asymptotic properties}
We first state the parameter space and some technical conditions.
Let ${\bs \theta}^{*}$ be the true parameter vector satisfying $-Q_n \le \alpha_i^* + \beta_j^* \le Q_n$ ($1\le i \neq j \le n$)
for variable $Q_n$. We consider the parameter space
\begin{equation}\label{parameter}
\Theta=\{{\bs \theta}: -Q_{n}-2r\leq\alpha_{i}+\beta_{j}\leq Q_{n}+2r,~ \text{for all}\  1 \leq i\neq j\leq n\},
\end{equation}
where $r=\|[F'({\bs \theta}^{*})]^{-1}F({\bs \theta}^{*})\|_{\infty}$. Let ${\bs g}=(d_{1},\ldots,d_{n},b_{1},\ldots,b_{n-1})^{\top}$, $g_{2n}=b_{n}$ and $\bs{\tilde{ g}}=(\tilde{d}_{1},\ldots,\tilde{d}_{n},\tilde{b}_{1},\ldots,\tilde{b}_{n-1})^{\top}$, $\tilde{g}_{2n}=\tilde{b}_{n}$.
By \cite{fan:2021}, the following condition is given.
When $\bs \theta \in \Theta$, for $i=1,\ldots, n$ and $j=1,\ldots, n-1$ with $j\neq i$, the derivatives of $\mu(\cdot)$ satisfy
\begin{equation}\label{eqq3}
 m\leq\Big|\cfrac{\partial \mu(\alpha_{i}+\beta_{j})}{\partial \alpha_{i}} \Big|= \Big|\cfrac{\partial \mu(\alpha_{i}+\beta_{j})}{\partial \beta_{j}}\Big|\leq M,
  \end{equation}
\begin{equation}\label{eqq4}
\max_{i,j}\{\Big|\cfrac{\partial^{2} \mu(\alpha_{i}+\beta_{j})}{\partial \alpha^{2}_{i}}\Big|, \Big|\cfrac{\partial^{2} \mu(\alpha_{i}+\beta_{j})}{\partial \beta_{j}\partial \alpha_{i}}\Big|,\Big|\cfrac{\partial^{2} \mu(\alpha_{i}+\beta_{j})}{\partial \beta^{2}_{j}}\Big|\}\leq \eta_{1},
\end{equation}
where $M$, $m$ and $\eta_{1}$ are functions on variable $Q_{n}$.

Now, we present the existence and consistency of $\bs{\widehat{\theta}}$ under some mild conditions.
Define the following system of functions:
\begin{equation}\label{eq:F-DP}
\begin{split}
 & F_i( \bs \theta ) =   \tilde{d}_i -  \sum_{k=1; k \neq i}^n \mu(\alpha_{i}+\beta_{k}), ~~  i=1, \ldots, n, \\
& F_{n+j}( \bs \theta ) =  \tilde{b}_j - \sum_{k=1; k\neq j}^n \mu(\alpha_{k}+\beta_{j}),  ~~  j=1, \ldots, n, \\
 &F(\bs  \theta ) =  (F_{1}(\bs\theta),\cdots,F_{n}(\bs\theta),\cdots,F_{2n-1}(\bs\theta))^{\top}.
\end{split}
\end{equation}
Obviously, the solution to the equation $F({\bs \theta})=0$ serves as precisely EDP estimator.
The consistency is proved by constructing a Newton iterative sequence: $\bs\theta^{(k+1)}=\bs\theta^{(k)}-[F'(\bs\theta^{(k)})]^{-1}F(\bs\theta^{(k)})$.
By choosing the true value $\bs{\theta^{*}}$ as the initial value, we derive the error between $\bs{\theta^{*}}$ and $\bs{\widehat{\theta}}$. We state the consistency of $\bs{\widehat{\theta}}$, whose proof is given in the Appendix.
\begin{theorem}\label{theorem1}
If inequalities (\ref{eqq3}) and (\ref{eqq4}) hold, and
\begin{eqnarray}
\label{eq2}
\frac{M^{2}}{nm^{3}}( \sqrt{n\log n } + \kappa\sqrt{\log n} ) =o(1), \\
\label{eq3}
\frac{M^{4}\eta_{1}}{nm^{6}}( \sqrt{n\log n } + \kappa\sqrt{\log n} )=o(1),
\end{eqnarray}
the EDP estimator $\bs{\widehat{\theta}}$ in (\ref{eqqq1}) satisfies
    \begin{equation}\label{eq4}
\parallel \bs{\widehat{\theta}}-\bs{\theta^{*}}\parallel_{\infty}=O_{p}\left(\frac{M^{2}}{nm^{3}}( \sqrt{n\log n } + \kappa\sqrt{\log n} )\right)=o_{p}(1).
    \end{equation}
\end{theorem}
Next we state the asymptotic normality of $\bs{\widehat{\theta}}$. Before introducing our main theorem, we need to provide some preliminary results on the covariance matrix of bi-degree sequences. Let $U= u_{i,j}=\text{Cov}\{\bs{g}-E(\bs{g})\}$. It follows that
$\mathrm{Var}(a_{i,j}) = u_{i,j}$ for $i=1, \ldots, n$ and $j=1, \ldots, n-1$.
Obviously, if
\begin{equation}
\label{eQ5}
m_{u}\leq \min_{i,j} \mathrm{Var}(a_{i,j}) \leq \max_{i,j} \mathrm{Var}(a_{i,j}) \leq M_{u},
\end{equation}
then $U \in \mathcal{L}_n(m_{u}, M_{u})$, $u_{2n,2n}=\mathrm{Var} (b_{n})$ and $(n-1)m_{u}\leq u_{i,i}\leq (n-1)M_{u}$. If $M_{u}/m_{u}=o(n)$, then $n\cdot m_{u}\rightarrow\infty$. So we have $u_{i,i}\rightarrow\infty$ as $n\rightarrow\infty$. Then we give the following lemma.
\begin{lemma}\label{lemma4}
Let $\lambda=\exp(-\epsilon/2)$ and $\kappa=2(-\log \lambda)^{-1}$.
(i) If the condition $\frac{\kappa (\log n)^{1/2}}{m_{u}}=o(1)$ and $M_{u}/m_{u}=o(n)$ hold, then as
$n$ tends to infinity, for any fixed $k \ge 1$, the vector
$[S\{\bs{\tilde{g}}-\E (\bs g) \}]_{1}, \cdots, [S\{\bs{\tilde{g}}-\E (\bs g) \}]_{k}$ converges in distribution to a multivariate normal distribution with mean $\mathbf{0}$
and covariance matrix $Z$, where $Z$ is given as follows:
\begin{equation}\label{equ2}
  \mathrm{diag}( \frac{u_{1,1}}{v_{1,1}^{2}}, \ldots, \frac{u_{k,k}}{v_{k,k}^{2}})+ (\frac{u_{2n,2n}}{v_{2n,2n}^{2}}) \mathbf{1}_k \mathbf{1}_k^\top,
\end{equation}
where $\mathbf{1}_k$ denotes a $k$-dimensional column vector with all elements equal to $1$.\\
(ii) Let
\[
s_n^2=\mathrm{Var}( \sum_{i=1}^n \gamma_i - \sum_{j=1}^{n-1} \zeta_j) = (2n-1)\frac{ 2\lambda}{ (1-\lambda)^2}.
\]
Assume that  $s_n/v_{2n,2n}^{1/2} $ converges to constant $c$. For any fixed $k \ge 1$, the vector $[S\{\bs{\tilde{g}}-\E (\bs g) \}]_{1}, \cdots, [S\{\bs{\tilde{g}}-\E (\bs g) \}]_{k}$ converges in distribution to a $k$-dimensional multivariate normal distribution with mean zero
and covariance matrix $B=(b_{i,j})_{k\times k}$, where $B$ is given as follows:
\begin{equation}\label{equ3}
\mathrm{diag}( \frac{u_{1,1}}{v_{1,1}^{2}}, \ldots, \frac{u_{k,k}}{v_{k,k}^{2}})+ (\frac{u_{2n,2n}}{v_{2n,2n}^{2}} + \frac{s_n^2}{v_{2n,2n}^2}) \mathbf{1}_k \mathbf{1}_k^\top,
\end{equation}
where $\mathbf{1}_k$ denotes a $k$-dimensional column vector with all elements equal to $1$.
\end{lemma}
We prove Lemma \ref{lemma4} in the Appendix. Finally, we state the central limit theorem of the EDP estimator $\bs {\widehat{\theta}}$, as shown below.
\begin{theorem}\label{theorem2}
If inequalities \eqref{eqq4},\ \eqref{eq4} hold and
\begin{equation}\label{eq5}
\frac{(n+2\kappa n^{1/2}+\kappa^{2})M^{6}\eta_{1}\log n}{m^{9}n^{3/2}}=o(1),
\end{equation}
(i) if the condition $\frac{\kappa (\log n)^{1/2}}{m_{u}}=o(1)$ holds,
then as $n$ tends to infinity, for any fixed $k \ge 1$, the vector
$(\bs{\widehat{\theta}}-\bs{\theta^{*}})_{1},\cdots,(\bs{\widehat{\theta}}-\bs{\theta^{*}})_{k}$ converges in distribution to a multivariate normal distribution with mean zero
and covariance matrix $Z$, defined in \eqref{equ2}.\\
(ii)
then for any fixed $k \ge 1$, the vector $(\bs{\widehat{\theta}}-\bs{\theta^{*}})_{1},\cdots,(\bs{\widehat{\theta}}-\bs{\theta^{*}})_{k}$ converges in distribution to a $k$-dimensional multivariate normal distribution with mean zero
and covariance matrix $B$, defined in \eqref{equ3}.
\end{theorem}
\begin{remark}
It is meaningful to compare the above theorem with Theorem 2.2 in \cite{fan:2021}. The key differences in that the variance of $\widehat{\theta}_i$ has an additional factor $s_n^2/v_{2n,2n}^2$. This is due to the fact that they only consider no differential private case. The asymptotic expression of $\widehat{\theta}_i$ contains a term $\sum_{i=1}^n \gamma_i - \sum_{j=1}^{n-1} \zeta_j$. Its variance is in the magnitude of $ne^{-\epsilon/2}$. When $\epsilon$ becomes small, the variance increases quickly, such that its impact on $\widehat{\theta}_i$ cannot be ignored when it increases to a certain level. This leads to the appearance of the additional variance factor.
\end{remark}
\section{Application}\label{s3}
In this section, we provide application of the unified theoretical result to the Probit model. The Probit model assume that $a_{i,j}$ is from Bernoulli distribution with the success probability:
\begin{equation}\label{Probit model}
  P(a_{i,j}=1)=\displaystyle\int_{-\infty}^{\alpha_{i}+\beta_{j}}\frac{1}{\sqrt{2\pi}}e^{-\frac{x^{2}}{2}}dx=\Phi(\alpha_{i}+\beta_{j}),
  \end{equation}
where $\Phi(\cdot)$ is the cumulative distribution function of the standard normal distribution.
In this case, we get $$\mu(\alpha_i+\beta_j)=\displaystyle\int_{-\infty}^{\alpha_{i}+\beta_{j}}\frac{1}{\sqrt{2\pi}}e^{-\frac{x^{2}}{2}}dx.$$
By direct calculations, we have
\[
  \frac{\partial\mu(\alpha_i+\beta_j)}{\partial \alpha_{i}}=\frac{\partial  \mu(\alpha_i+\beta_j)}{\partial \beta_{j}}=\frac{1}{\sqrt{2\pi}}e^{-\frac{(\alpha_{i}+\beta_{j})^{2}}{2}}.
\]
The function $h(x)=\exp(-x^2/2)$ is symmetric on $x$ and is a decreasing function when $x\geq 0$. So we have the following inequalities,
\[
\frac{1}{\sqrt{2\pi}}e^{-\frac{(Q_{n}+2r)^{2}}{2}}\leq \Big|\frac{\partial   \mu(\alpha_i+\beta_j)}{\partial \alpha_{i}} \Big|\leq \frac{1}{\sqrt{2\pi}},
\frac{1}{\sqrt{2\pi}}e^{-\frac{Q_{n}^{2}}{2}}\leq \Big|\frac{\partial   \mu(\alpha_i+\beta_j)}{\partial \alpha^{*}_{i}} \Big|\leq \frac{1}{\sqrt{2\pi}}.
    \]
Thus $F'(\bs{\theta}^{*})\in \mathcal{L}_n(m^{*}, M^*),$ where $M^*=\frac{1}{\sqrt{2\pi}}, m^*=\frac{1}{\sqrt{2\pi}}e^{-\frac{Q_{n}^{2}}{2}}.$ Following \cite{Yan:2021} and \cite{Wang:2020}, we gain the following inequality holds,
\begin{equation}\label{eq:con4}
\max\{ \max_i|{d}_i - \E (d_i) |, \max_j | {b}_j - \E( b_j)| \} \leq \sqrt{n\log n } + \kappa\sqrt{\log n} .
\end{equation}
Again, by (\ref{eq:con4}) and Lemma 1 of \cite{Yan:Leng:Zhu:2016}, if $\kappa=o_{p}(n^{1/2})$, we obtain
\begin{equation*}
r=O\left({e^{3Q_{n}^{2}/2}}\sqrt\frac{ \log n}{n}\right).
\end{equation*}
If $e^{3Q_{n}^{2}/2}=o((n/\log n)^{1/2})$, then $r\rightarrow 0$ as $n$ tends to infinity. Thus $r$ can be small enough to ignore, for any $\bs{\theta}$ belonging to the set $\Omega( \bs{\theta}^*, 2r),$ we derive $F'(\bs{\theta})\in \mathcal{L}_n(m, M),$ where
  \begin{equation}\label{m1}
    \begin{split}
  M=\frac{1}{\sqrt{2\pi}}, m=\frac{1}{\sqrt{2\pi}}e^{-\frac{Q_{n}^{2}}{2}}.
\end{split}
\end{equation}
On the other hand, for any $i=1,\ldots, n$ and $j=1,\ldots, n-1$ with $j\neq i$,
\[
    \begin{split}
\frac{\partial^{2}  \mu(\alpha_i+\beta_j)}{\partial \alpha^{2}_{i}}=\frac{\partial^{2} \mu(\alpha_i+\beta_j)}{\partial \beta^{2}_{j}}=\frac{\partial^{2}  \mu(\alpha_i+\beta_j)}{\partial \alpha_{i}\partial\beta_{j}}=\frac{\partial^{2} \mu(\alpha_i+\beta_j)}{\partial\beta_{j}\partial\alpha_{i}}=-\cfrac{(\alpha_{i}+\beta_{j})e^{-\frac{(\alpha_{i}+\beta_{j})^{2}}{2}}}{\sqrt{2\pi}}.
\end{split}
    \]
The function $g(x)=-x\exp(-x^2/2)$ is decreasing function when $-1<x<1$; otherwise, it is the increasing function. Moreover, the function value is negative when $x>0$ and the function value is positive when $x<0$. So, the function value reaches the maximum $e^{-1/2}$ when $x=-1$. Then, we have
$$\Big|-\cfrac{\alpha_{i}+\beta_{j}}{\sqrt{2\pi}}e^{-\frac{(\alpha_{i}+\beta_{j})^{2}}{2}}\Big|
\leq \frac{1}{\sqrt{2\pi e}}.$$
Thus, $\eta_{1}=\frac{1}{\sqrt{2\pi e}}.$
If $e^{Q^{2}_{n}}=o((n/\log n)^{1/6})$ and $\kappa=o_{p}(n^{1/2})$,
\[
    \begin{split}
&\frac{M^{4}\eta_{1}}{nm^{6}}( \sqrt{n\log n } + \kappa\sqrt{\log n} )
=O\left({e^{3Q_{n}^{2}}} \sqrt\frac{ \log n}{n}\right)=o(1),
\end{split}
    \]
where $M$ and $m$ are given in (\ref{m1}), then (\ref{eq4}) is satisfied. By Theorem \ref{theorem1}, the consistency of $\bs{\widehat{\theta}}$ is as follows.
\begin{corollary}
{ If $e^{Q^{2}_{n}}=o((n/\log n)^{1/6})$ and $\kappa=o_{p}(n^{1/2})$,
 the EDP estimator $\bs{\widehat{\theta}}$ satisfies
 \begin{equation}
\parallel \bs{\widehat{\theta}}-\bs{\theta^{*}}\parallel_{\infty}=O_{p}\left({e^{3Q_{n}^{2}/2}}\sqrt\frac{ \log n}{n}\right)=o_{p}(1).
\end{equation}}
\end{corollary}
Again, note that both $d_{i}=\sum_{j\neq i}a_{i,j}$ and $b_{j}=\sum_{j\neq i}a_{i,j}$ are sums of $n-1$ independent random variables, respectively. It can be shown that $U=\text{Cov}{\{{\bs{g}}-\E(\bs{g})\}}\in \mathcal{L}_n(m_{u}, M_{u})$, where
$$m_{u}=\Phi(-Q_{n})(1-\Phi(-Q_{n})), M_{u}=1/4.$$ Since $\Phi(x)(1-\Phi(x))$ is an increasing function for $x$ when $x\leq\Phi^{-1}(1/2)$ and a decreasing function when $x\geq\Phi^{-1}(1/2),$ we have
$$(n-1)\Phi(-Q_{n})(1-\Phi(-Q_{n}))\leq u_{i,i}\leq \frac{n-1}{4}, \ i=1,\cdots,2n.$$
If $e^{Q_{n}^{2}/2}=o(n^{1/18}/(\log n)^{1/9})$ and $\kappa=o_{p}(n^{1/2})$, then
\[
    \begin{split}
\frac{M^{6}\eta_{1}\log n(n+2\kappa n^{1/2}+\kappa^{2})}{m^{9}n^{3/2}}=O\left(\frac{e^{9Q_{n}^{2}/2}\log n}{n^{1/2}}\right)=o(1).
\end{split}
    \]
By Theorem \ref{theorem2} and Lemma \ref{lemma4}, the central limit theorem of $\bs{\widehat{\theta}}$ is as below.
\begin{corollary}\label{col4}
{ Assume that $e^{Q_{n}^{2}/2}=o(n^{1/18}/(\log n)^{1/9})$ and $\kappa=o_{p}(n^{1/2})$.\\
(i) If the condition $\frac{\kappa (\log n)^{1/2}}{m_{u}}=o(1)$ holds, then as $n$ tends to infinity, for any fixed $k \ge 1$, the vector $(\bs{\widehat{\theta}}-\bs{\theta^{*}})_{1},\cdots,(\bs{\widehat{\theta}}-\bs{\theta^{*}})_{k}$ converges in distribution to a multivariate normal distribution with mean of $\mathbf{0}$ and covariance matrix of $Z$.\\
(ii) Let
\[
s_n^2=\mathrm{Var}( \sum_{i=1}^n \gamma_i - \sum_{j=1}^{n-1} \zeta_j) = (2n-1)\frac{ 2\lambda}{ (1-\lambda)^2}.
\]
Assume that  $s_n/v_{2n,2n}^{1/2} $ converges to constant $c$. For any fixed $k \ge 1$, the vector $(\bs{\widehat{\theta}}-\bs{\theta^{*}})_{1},\cdots,(\bs{\widehat{\theta}}-\bs{\theta^{*}})_{k}$ converges in distribution to a
$k$-dimensional multivariate normal distribution with mean of $\mathbf{0}$ and covariance matrix of $B$.
}
\end{corollary}
\section{Simulations}
\label{section:simulation}
\par In this section, we evaluate the asymptotic results for Probit model (\ref{Probit model}) of directed graph model through numerical simulations. Similar to \cite{Yan:Leng:Zhu:2016}, we determine the parameter values using a linear form. In particular, for $i=0,\ldots,n-1$, we assign $\alpha^{*}_{i+1}=(n-1-i)L/(n-1)$; for simplicity, we let $\beta^{*}_{i}=\alpha^{*}_{i}$ for $i=1,\ldots,n-1$ with $\beta^{*}_{n}=0$. We take into account three distinct values for $L$, $L=(\log n)^{1/2}$, $\log(\log n)$ and $0$, respectively. We simulate two different values for $\epsilon$, $\epsilon=2$ and $\log n/n^{1/4}$. We carry out simulations under two different sizes of networks: $n=100$ and $n=200.$ Under each simulation setting, $10, 000$ datasets are generated.

\par By Corollary \ref{col4},
\begin{align*}
   \hat{\xi}_{i,j}& =[\hat{\alpha}_{i}-\hat{\alpha}_{j}-(\alpha_{i}^{*}-\alpha_{j}^{*})]/(\hat{z}_{i,i}+\hat{z}_{j,j})^{1/2}, \\
  \hat{\zeta}_{i,j} & =[\hat{\alpha}_{i}+\hat{\beta}_{j}-(\alpha_{i}^{*}+\beta_{j}^{*})]/(\hat{z}_{i,i}+\hat{z}_{n+j,n+j})^{1/2},\\
  \hat{\eta}_{i,j}& =[\hat{\beta}_{i}-\hat{\beta}_{j}-(\beta_{i}^{*}-\beta_{j}^{*})]/(\hat{z}_{n+i,n+i}+\hat{z}_{n+j,n+j})^{1/2}
\end{align*}
are all asymptotically distributed as standard normal random variables, where $\hat{z}_{i,i}$ is the estimate of $z_{i,i}$ by replacing ${\theta}_{i}^{*}$ with ${\hat{\theta}}_{i}.$ Three particular pairs $(1,2), (n/2,n/2+1)$ and $(n-1, n)$ are discussed for the pair $(i,j)$.

\par We apply the quantile-quantile (QQ) plot to demonstrate the normality of $\hat{\xi}_{i,j},\hat{\zeta}_{i,j},\hat{\eta}_{i,j}$. Given the similarity in the results, thus we here only show the QQ-plot for $\hat{\xi}_{i,j}$ in Figure \ref{f1} under the case of $n=100$ and $\epsilon=2$. This figure reveals that the empirical quantiles closely align with
the theoretical quantiles.

\begin{figure}\centering
\includegraphics[ height=3in, width=4in, angle=0]{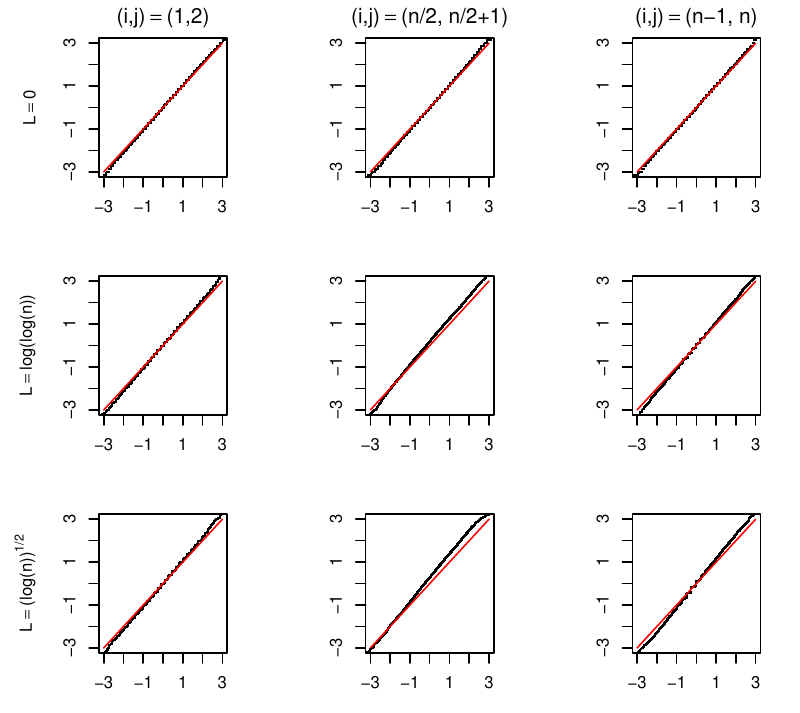}
\caption{The QQ plot of $\hat{\xi}_{i,j}$. The theoretical and empirical quantiles are represented by the horizontal and vertical axes, respectively, and the red lines serve as reference lines $y=x.$ }
\label{f1}
\end{figure}

Table \ref{Table 1} summarizes the simulation results for $\alpha_i-\alpha_j$. As expected, the length of the confidence interval decreases with $n$ and increases with $L$.

\begin{table}\centering
\caption{Estimated coverage frequency ($\times 100\%$) of $\alpha_i-\alpha_j$ as well as the length of confidence intervals (in square brackets), and the frequency ($\times 100\%$) that the estimates do not exist (in parentheses).}
\label{Table 1}
\vspace{0.5em}
\scalebox{0.8}{
		\def\arraystretch{1.0}
	\setlength{\tabcolsep}{1.2mm}{
\begin{tabular}{ccccccccccccccccccccccc}
\hline
n         &  $(i,j)$   & $L=0$   &  $L=\log(\log n)$   &  $L=(\log n)^{1/2}$ \\
\hline
 \multicolumn {5}{c}{$\epsilon$=2}\\
\hline
100         &$(1,2) $&$  93.80[0.36](0)  $&$ 93.61[0.41] (0.06) $&$ 92.63[0.46](1.13)  $\\
            &$(50,51) $&$  93.49[0.36](0)  $&$ 92.78[0.48](0.06)  $&$ 91.28[0.57](1.13) $\\
            &$(99,100)$&$  93.96[0.36](0) $&$ 92.73[0.54] (0.06)$&$ 91.14[0.64] (1.13) $ \\

&&&&&&\\
200         &$(1,2)   $&$ 94.32[0.25](0) $&$ 94.00[0.29] (0) $&$94.29[0.33] (0)$\\
            &$(100,101)$&$ 94.64[0.25](0) $&$ 93.61[0.35](0) $&$92.97[0.41] (0)$ \\
            &$(199,200) $&$ 94.66[0.25](0) $&$ 93.94[0.39](0) $&$93.02[0.46](0)$ \\
\hline
\multicolumn {5}{c}{$\epsilon=\log n/n^{1/4}$}\\
\hline
100         &$(1,2) $&$  92.37[0.36](0)  $&$ 92.40[0.41] (0.49) $&$ 90.07[0.46](4.46)  $\\
            &$(50,51) $&$  92.43[0.36](0)  $&$ 90.37[0.48](0.49)  $&$ 87.70[0.57](4.46) $\\
            &$(99,100)$&$  92.58[0.36](0) $&$ 89.92[0.54] (0.49)$&$ 88.06[0.65] (4.46) $ \\

&&&&&&\\
200         &$(1,2)   $&$ 94.19[0.25](0) $&$ 93.31[0.30] (0) $&$93.16[0.33] (0.07)$\\
            &$(100,101)$&$ 93.92[0.25](0) $&$ 92.47[0.35](0) $&$90.18[0.41] (0.07)$ \\
            &$(199,200) $&$ 93.71[0.25](0) $&$ 91.56[0.39](0) $&$91.14[0.46](0.07)$ \\
\hline
\end{tabular}}}
\end{table}

\emph {A data example.}
We evaluate the use of the proposed estimator on the Lazega's friendship network data (\cite{lazega2001collegial,snijders2006new}).
This dataset contains $71$ nodes, and each node represents an attorney.
A directed edge from one attorney to another exists if the former considers the later a friend outside of work.
In total, there are $575$ directed edges.
In-degrees and out-degrees vary from 0 to 22 and 25, respectively.

To guarantee the existence of the nonprivate MLE, we removed eight nodes before analysis, where two nodes have zero
in- and out- degrees, four nodes have zero out-
degrees and non-zero in-degrees and two nodes have zero in-
degrees and non-zero out-degrees. We chose the privacy parameter $\epsilon$ as 1. Table \ref{Table 2} reports the estimators of $\alpha_{i}$ and $\beta_{j}$ with their estimated standard errors, and their corresponding in bi-degree sequences. We can see that no extremes are observed at all nodes, both in terms of the estimates themselves and their standard errors.

\begin{table}\centering
\caption{Lazega's friendship network dataset: the edge differentially private estimator of ${\alpha}_{i}$ and ${\beta}_{j}$ with their standard errors and their corresponding bi-degree sequence.}
\label{Table 2}
\vspace{0.5em}
\scalebox{0.7}{
		\def\arraystretch{0.9}
	\setlength{\tabcolsep}{0.6mm}{
\begin{tabular}{ccccccccccccccccccccccc}
\hline
Vertex   &$\hat{\alpha}_{i}$ &$\sqrt{\hat{\sigma}_{i,i}}$  &$\hat{\beta}_{j}$ &$\sqrt{\hat{\sigma}_{j,j}} $ &$d_{i}$  &$b_{j}$ &Vertex   &$\hat{\alpha}_{i}$  &$\sqrt{\hat{\sigma}_{i,i}}$  &$\hat{\beta}_{j}$ &$\sqrt{\hat{\sigma}_{j,j}} $&$d_{i}$ &$b_{j}$ &Vertex   &$\hat{\alpha}_{i}$  &$\sqrt{\hat{\sigma}_{i,i}}$  &$\hat{\beta}_{j}$ &$\sqrt{\hat{\sigma}_{j,j}} $&$d_{i}$ &$b_{j}$\\
\hline
\textbf{1}  &$-1.61$ &$0.07$ &$0.10$ &$0.06$ &$4$  &$5$ &\textbf{22} &$-0.58$ &$0.03$ &$0.79$  &$0.03$&$23$  &$17$   &  \textbf{43} &$-1.62$ &$0.05$&$-0.39$  &$0.07$&$7$  &$4$\\
\textbf{2}  &$-1.84$ &$0.07$ &$0.34$ &$0.04$ &$4$  &$10$&\textbf{23} &$-1.21$&$0.04$  &$0.35$  &$0.04$&$11$  &$10$& \textbf{44} &$-1.85$&$0.08$ &$0.00$  &$0.06$ &$4$  &$6$\\
\textbf{3}  &$-0.97$ &$0.03$ &$0.61$&$0.03$  &$15$  &$14$&\textbf{24}  &$-1.42$ &$0.04$ &$0.97$  &$0.03$&$9$  &$22$ & \textbf{45} &$-1.43$&$0.07$ &$0.19$  &$0.07$ &$8$  &$8$ \\
 \textbf{4}&$-1.85$ &$0.08$ &$0.00$  &$0.06$&$3$  &$5$ & \textbf{25}  &$-1.08$  &$0.04$&$0.78$ &$0.03$ &$13$  &$17$ & \textbf{46} &$-1.52$&$0.05$&$0.19$  &$0.05$ &$6$  &$7$ \\
\textbf{5} &$-2.20$ &$0.11$ &$-0.60$ &$0.11$ &$2$  &$2$&\textbf{26}  &$-1.09$ &$0.04$ &$0.35$  &$0.04$&$12$  &$9$ & \textbf{47} &$-1.20$&$0.04$ &$0.61$  &$0.03$ &$11$  &$14$\\
\textbf{6} &$-2.19$ &$0.43$ &$0.18$ &$0.05$ &$1$  &$7$  &\textbf{27}  &$-1.28$ &$0.05$ &$0.35$ &$0.05$ &$10$  &$10$ & \textbf{48}&$-1.51$ &$0.05$&$0.42$   &$0.04$ &$7$  &$11$\\
\textbf{7} &$-1.71$ &$0.06$ &$0.54$ &$0.03$ &$6$  &$14$  &\textbf{28}  &$-1.62$&$0.05$  &$-0.11$  &$0.06$&$6$  &$5$&\textbf{49}&$-1.43$ &$0.04$&$0.35$  &$0.04$  &$8$  &$10$\\
\textbf{8}  &$-1.04)$ &$0.03$ &$-0.22$  &$0.07$&$14$  &$4$& \textbf{29}  &$-0.45$ &$0.03$ &$0.74$  &$0.03$&$25$  &$15$   &\textbf{50} &$-1.35$&$0.04$ &$0.48$   &$0.04$ &$9$  &$12$  \\
\textbf{9} &$-1.71$&$0.05$  &$0.60$  &$0.03$&$5$  &$14$ &\textbf{30}  &$-1.72$ &$0.07$ &$0.19$&$0.05$  &$4$  &$7$ & \textbf{51} &$-1.03$ &$0.03$&$0.55$    &$0.04$&$13$  &$12$ \\
\textbf{10} &$-0.73$ &$0.03$ &$0.02$ &$0.04$ &$22$  &$8$ &\textbf{31}  &$-1.22$  &$0.04$&$-0.59$ &$0.08$&$12$  &$3$  &    \textbf{52} &$-1.85$ &$0.06$&$-0.39$  &$0.07$ &$5$  &$4$   \\
\textbf{11} &$-1.08$  &$0.03$&$0.88$ &$0.03$ &$14$  &$20$  &\textbf{32}  &$-1.84$ &$0.05$ &$0.27$  &$0.04$&$6$  &$11$ &\textbf{53} &$-1.84$&$0.07$ &$0.19$  &$0.05$  &$4$  &$8$\\
\textbf{12}  &$-1.43$ &$0.04$ &$0.01$ &$0.05$ &$8$  &$6$  &\textbf{33} &$-1.35$  &$0.04$&$0.35$ &$0.04$  &$9$  &$10$   & \textbf{54} &$-1.85$ &$0.08$&$-0.24$ &$0.08$ &$3$  &$3$\\
\textbf{13} &$-1.73$ &$0.08$ &$-0.39$ &$0.17$&$4$  &$2$ &\textbf{34}  &$-1.28$ &$0.04$ &$0.48$ &$0.04$ &$9$   &$11$  &\textbf{55} &$-2.00$ &$0.07$&$-0.24$ &$0.06$&$4$  &$5$ \\
\textbf{14} &$-1.43$  &$0.04$&$0.35$  &$0.04$&$8$  &$10$ &\textbf{35} &$-1.61$&$0.05$  &$0.41$ &$0.04$  &$8$  &$13$&\textbf{56} &$-0.71$ &$0.03$&$0.73$ &$0.03$  &$19$  &$15$\\
\textbf{15}  &$-0.56$&$0.03$  &$0.84$ &$0.03$ &$23$  &$18$ &\textbf{36}&$-1.43$ &$0.04$&$0.54$ &$0.03$    &$8$  &$13$& \textbf{57} &$-0.54$ &$0.03$&$0.37$   &$0.05$&$22$  &$8$\\
\textbf{16}  &$-1.44$ &$0.04$ &$-0.23$ &$0.05$ &$9$  &$5$ &\textbf{37} &$-1.28$ &$0.04$&$0.19$  &$0.05$   &$10$  &$8$ &\textbf{58} &$-0.98$ &$0.03$&$-0.58$  &$0.17$ &$16$  &$3$\\
\textbf{17}  &$-2.00$ &$0.07$ &$-0.39$ &$0.07$ &$4$  &$4$ &\textbf{38} &$-1.20$ &$0.04$&$0.72$   &$0.03$  &$12$  &$17$&\textbf{59} &$-2.20$ &$0.07$&$-0.92$ &$0.08$&$4$  &$3$\\
\textbf{18} &$-1.09$ &$0.03$ &$0.20$ &$0.05$ &$12$  &$7$ &\textbf{39} &$-0.92$ &$0.03$&$0.36$  &$0.04$   &$15$  &$9$ &\textbf{60} &$-1.52$ &$0.05$&$0.01$  &$0.06$ &$6$  &$5$\\
\textbf{19} &$-1.27$ &$0.05$ &$0.77$ &$0.03$ &$8$  &$15$ &\textbf{40} &$-0.86$&$0.03$ &$0.67$ &$0.03$    &$15$  &$13$ &\textbf{61} &$-1.72$&$0.07$ &$-0.23$ &$0.08$ &$5$  &$4$ \\
\textbf{20}  &$-1.44$ &$0.05$ &$0.01$ &$0.06$ &$8$  &$6$ &\textbf{41}&$-1.62$ &$0.05$&$-0.23$ &$0.07$  &$6$  &$4$ &\textbf{62}&$-1.52$ &$0.04$&$-0.10$ &$0.05$ &$7$  &$5$ \\
\textbf{21}  &$-2.51$ &$0.11$ &$0.10$ &$0.05$ &$1$  &$7$ &\textbf{42}&$-1.85$&$0.08$ &$0.00$  &$0.06$&$3$  &$5$&\textbf{63} &$-2.51$ &$0.11$&$0.00$   &$0.05$ &$1$  &$6$  \\
\hline
\end{tabular}}}
\end{table}

\section{Discussion}
\label{section-discussion}
We have established the asymptotic theory in a class of directed random graph model parameterized by the differentially private bi-sequence  and illustrated application to the Probit model. The result shows that statistical inference can be made using the noisy bi-sequence. We assume that the edges are mutually independent in this work. We should be able to obtain consistent conclusion if the edges are dependent, provided that the conditions stated in Theorem \ref{theorem1} are met. However, the asymptotic normality of the estimator is not clear. To avoid this problem, we need appropriately select a probability distribution for directed random graphs when using the existing method. In the further, we may relax our theoretical conditions to ignore the independence of edges.
\section*{Appendix}\label{Appendix}
\renewcommand\thesection{\Alph{section}}
\setcounter{section}{+1}
\renewcommand\theequation{A.\arabic{equation}}
We start with some notations. Let $\R_+ = (0, \infty)$, $\R_0=[0, \infty)$,
$\N=\{1,2,\ldots \}$, $\N_0 = \{0,1,2,\ldots\}$.
Given a vector $\mathbf{x}=(x_1, \ldots, x_n)^\top\in \R^n$, let
$\|\mathbf{x}\|_\infty = \max_{1\le i\le n} |x_i|$ represent the $\ell_\infty$-norm of $\mathbf{x}$.
For matrix $J=(J_{i,j})_{n\times n}$,
$\|J\|_\infty$ denotes the matrix norm induced by the $\|\cdot \|_\infty$-norm on vectors in $\R^n$:
\[
\|J\|_\infty = \max_{\mathbf{x}\neq 0} \frac{ \|J\mathbf{x}\|_\infty }{\|\mathbf{x}\|_\infty}=\max_{1\le i\le n}\sum_{j=1}^n |J_{i,j}|.
\]
For a matrix $A=\left(a_{i, j}\right)$, $\|A\| :=\max _{i, j}\left|a_{i, j}\right|$ denotes the matrix maximum norm.
\par Given $m,M>0,$ we say a $(2n-1)\times (2n-1)$ matrix $V=(v_{i,j})$ belongs to the matrix class $\mathcal{L}_{n}(m,M)$ if $V$ satisfies the following conditions:
\begin{equation*}
\begin{array}{l}%
m\leq v_{i,i}-\sum_{j=n+1}^{2n-1}v_{i,j}\leq M,i=1,\dots,n-1;~~v_{n,n}=\sum_{j=n+1}^{2n-1}v_{n,j},\\
v_{i,j}=0,i,j=1,\dots,n,i\neq j,\\
v_{i,j}=0,i,j=n+1,\dots,2n-1,i\neq j,\\
m\leq v_{i,j}=v_{j,i}\leq M,i=1,\dots,n,j=n+1,\dots,2n-1,j\neq n+i,\\
v_{i,n+i}=v_{n+i,i}=0,i=1,\dots,n-1,\\
v_{i,i}=\sum_{k=1}^{n}v_{k,i}=\sum_{k=1}^{n}v_{i,k},i=n+1,\dots,2n-1.
\end{array}
\end{equation*}
Obviously, if $V \in \mathcal{L}_n(m, M),$ then $V$ is a $(2n-1)\times (2n-1)$ diagonally dominant, symmetric nonnegative matrix and $V$ has the following structure:
\begin{equation*}
V=\left(
\begin{array}{ccccccccccc}
V_{11}\ \  V_{12} \\
V^{\top}_{12} \ \ V_{22}
\end{array}
\right),
\end{equation*}
where $n\times n$ matrix $V_{11}$ and $(n-1)\times (n-1)$ matrix $V_{22}$ are diagonal matrices, $V_{12}$ is a nonnegative matrix whose nondiagonal elements are positive and diagonal elements equal to zero.\par
Define $v_{2n,i}=v_{i,2n}:=v_{i,i}-\sum_{j=1;j\neq i}^{2n-1}v_{i,j},$ for $i=1,\ldots,2n-1$ and $v_{2n,2n}=\sum_{i=1}^{2n-1}v_{2n,i}.$ Then $m\leq v_{2n,i}\leq M$ for $i=1,\ldots,n-1, v_{2n,i}=0$ for $i=n,n+1,\ldots,2n-1$ and $v_{2n,2n}=\sum_{i=1}^{n}v_{i,2n}=\sum_{i=1}^{n}v_{2n,i}.$
Generally, the inverse of $V$, $V^{-1}$, does not possess a closed-form expression. \cite{Yan:Leng:Zhu:2016} proposed a matrix $S=(s_{i,j})$ to approximate $V^{-1}$,
 \begin{equation*}
S=\left(
\begin{array}{ccccccccccc}
S_{11}\ \  S_{12} \\
S^{\top}_{12} \ \ S_{22}
\end{array}
\right),
\end{equation*}
where $n\times n$ matrix $S_{11}=1/v_{2n,2n}+\text {diag}(1/v_{1,1},1/v_{2,2},\ldots,1/v_{n,n}),$ $S_{12}$ is an $n\times (n-1)$ matrix whose elements are all equal to $-1/v_{2n,2n}$ and  $(n-1)\times (n-1)$ matrix $S_{22}=1/v_{2n,2n}+\text {diag}(1/v_{n+1,n+1},1/v_{n+2,n+2},\ldots, 1/v_{2n-1,2n-1}).$

\subsection{ Proof of Theorem \ref{theorem1}}
To show the Theorem \ref{theorem1}, we need one lemma below.
\begin{lemma}\label{lemma3}(Lemma 2 in \cite{Yan:2021})
With probability approaching one,
\begin{equation}\label{equ1}
  \max\{\max_{i=1,\ldots,n}|d_i -E(d_i )|, \max_{j=1,\ldots,n}|b_j -E(b_j )|\}\leq \sqrt{n\log n } + \kappa\sqrt{\log n} ,
\end{equation}
where $\kappa=2(-\log \lambda)^{-1}=4/\epsilon$.
  \end{lemma}
\par \noindent{{\bf Proof of Theorem \ref{theorem1}.}} In the Newton's iterative step, we set the initial value $\bs{\theta}^{(0)}:=\bs{\theta}^{*}\in \Theta.$ Let
$F(\bs{\theta})=(F_{1}(\bs\theta),\ldots,F_{n}(\bs\theta),F_{n+1}(\bs\theta),\ldots,F_{2n-1}(\bs\theta))^{\top}$. By (\ref{eq:F-DP}), for $i=1,\ldots,n$
\[
    \begin{split}
\cfrac{\partial F_{i}}{\partial \alpha_{l}}&=0, \ \ \ \  \  l=1,\ldots, n, l\neq i, \ \ \ \ \cfrac{\partial F_{i}}{\partial\alpha_{i}}=-\sum_{k=1,k\neq i}^{n}\cfrac{\partial\mu(\alpha_{i}+\beta_{k})}{\partial\alpha_{i}},\\
\cfrac{\partial F_{i}}{\partial\beta_{j}}&=-\cfrac{\partial\mu(\alpha_{i}+\beta_{j})}{\partial\beta_{j}},\ \  \ \ j=1,\ldots, n-1, j\neq i,\ \ \ \ \cfrac{\partial F_{i}}{\partial\beta_{i}}=0,\\
\text{and  for}\  j=1, \ldots, n-1,\\
\cfrac{\partial F_{n+j}}{\partial \alpha_{l}}&=-\cfrac{\partial\mu(\alpha_{l}+\beta_{j})}{\partial\alpha_{l}}, \ \   \ \ l=1,\ldots, n, l\neq j,
\ \ \ \ \cfrac{\partial F_{n+j}}{\partial\alpha_{j}}=0.\\
\cfrac{\partial F_{n+j}}{\partial \beta_{l}}&=0, \ \ \ \  l=1,\ldots, n-1, l\neq j\ \ \ \ \cfrac{\partial F_{n+j}}{\partial\beta_{j}}=-\sum_{k=1,k\neq j}^{n}\cfrac{\partial\mu(\alpha_{k}+\beta_{j})}{\partial\beta_{j}}.
\end{split}
    \]
By (\ref{eqq3}), we get $-F'(\bs{\theta})\in \mathcal{L}_{n}(m,M)$.
To apply Theorem 7 of \cite{Yan:Leng:Zhu:2016}, it is necessary to compute the values of $r$ and $\rho r$ within the context of this theorem.
By Lemma 1 of \cite{Yan:Leng:Zhu:2016}, we have
\[
    \begin{split}
r&=\|[F'(\bs{\theta^{*}})]^{-1}F(\bs{\theta^{*}})\|_{\infty}\\
&\leq \cfrac{2c_{1}(2n-1)M^{2}}{m^{3}(n-1)^{2}}\|F(\bs{\theta^{*}})\|_{\infty}+\frac{|F_{2n}(\bs{\theta^{*}})|}{v_{2n,2n}}
+\max_{i=1,\ldots,2n-1}\frac{|F_{i}(\bs{\theta^{*}})|}{v_{i,i}}\\
&\leq \left(\cfrac{2c_{1}(2n-1)M^{2}}{m^{3}(n-1)^{2}}+\frac{1}{(n-1)m}+\frac{1}{m(n-1)}\right)(\sqrt{n\log n}+\kappa\sqrt{\log n})\\
&=\cfrac{M^{2}}{m^{3}(n-1)}\left(\frac{2c_{1}(2n-1)}{n-1}+\frac{2m^{2}}{M^{2}}\right)(\sqrt{n\log n}+\kappa\sqrt{\log n})\\
&\leq \frac{c_{2}M^{2}}{nm^{3}}(\sqrt{n\log n}+\kappa\sqrt{\log n}),
\end{split}
    \]
where $c_{2}$ represents a constant that is independent of $M$, $m$ and $n$. Let
$$F'_{i}(\bs{\theta})=(F'_{i,1}(\bs \theta),\cdots,F'_{i,n}(\bs \theta),\cdots,F'_{i,n+1}(\bs \theta),\cdots,F'_{i,2n-1}(\bs \theta))^{\top}:=\left(\cfrac{\partial F_{i}}{\partial\alpha_{1}},\ldots,\cfrac{\partial F_{i}}{\partial\alpha_{n}},\cfrac{\partial F_{i}}{\partial\beta_{1}},\ldots, \cfrac{\partial F_{i}}{\partial\beta_{n-1}}\right)^{\top}.$$
Then, for $i=1,\ldots,n,$ we have
\[
    \begin{split}
&\cfrac{\partial^{2} F_{i}}{\partial \alpha_{s}\partial \alpha_{l}}=0, \   s\neq l, \ \cfrac{\partial^{2} F_{i}}{\partial\alpha^{2}_{i}}=\sum_{k\neq i}\cfrac{\partial^{2}\mu(\alpha_{i}+\beta_{k})}{\partial\alpha^{2}_{i}},\\
&\cfrac{\partial^{2} F_{i}}{\partial \beta_{s}\partial \alpha_{i}}=\cfrac{\partial^{2}\mu(\alpha_{i}+\beta_{s})}{\partial\beta_{s}\partial\alpha_{i}}, \   s=1, \ldots, n-1, s\neq i, \cfrac{\partial^{2} F_{i}}{\partial\beta_{i}\partial\alpha_{i}}=0,\\
&\cfrac{\partial^{2} F_{i}}{\partial \beta^{2}_{j}}=\cfrac{\partial^{2}\mu(\alpha_{i}+\beta_{j})}{\partial\beta^{2}_{j}} ,\   j=1, \ldots, n-1,  \cfrac{\partial^{2} F_{i}}{\partial\beta_{s}\partial\beta_{l}}=0, s\neq l.
\end{split}
    \]
By the mean value theorem for vector-valued functions [\cite{Lang:1993} (page 341)], we have
$$F'_{i}(\textbf{x})-F'_{i}(\textbf{y})=J^{(i)}(\textbf{x}-\textbf{y}),\ {\textbf{x}, \textbf{y}}\in \Theta,$$
where $J^{(i)}_{s,l}=\displaystyle\int_{0}^{1}\cfrac{\partial F_{i,s}'}{\partial \theta_{l}}(t\textbf{x}+(1-t)\textbf{y})\text{dt}, \ s,l=1,\ldots,2n-1$. By (\ref{eqq4}), we get $$\max_{s}\sum_{l}|J^{(i)}_{s,l}|\leq 2\eta_{1}(n-1), \ \
\sum_{s,l}|J^{(i)}_{s,l}|\leq 4\eta_{1}(n-1).$$
Similarly, for $i=n+1,\ldots,2n-1,$ we also have $F'_{i}(\text{x})-F'_{i}(\text{y})=J^{(i)}(\textbf{x}-\textbf{y})$, and
the above inequalities. Correspondingly, for any $i$, we get
\[
    \begin{split}
 \|F'_{i}(\textbf{x})-F'_{i}(\textbf{y})\|_{\infty}&\leq \|J^{(i)}\|_{\infty}\|\textbf{x}-\textbf{y}\|_{\infty}\leq
2\eta_{1}(n-1)\|\textbf{x}-\textbf{y}\|_{\infty},
\end{split}
    \]
and for any vector $\textbf{v} \in \mathbb{R}^{2n-1},$
\[
    \begin{split}
\|[F'(\textbf{x})-F'(\textbf{y})]\textbf{v}\|_{\infty}&\leq \|\textbf{x}-\textbf{y}\|_{\infty} \|\textbf{v}\|_{\infty}\sum_{s,l}|J^{(i)}_{s,l}|\leq 4\eta_{1}(n-1)\|\textbf{x}-\textbf{y}\|_{\infty} \|\textbf{v}\|_{\infty}.
\end{split}
    \]
Thus, we can choose $K_{1}=4\eta_{1}(n-1)$ and $K_{2}=2\eta_{1}(n-1)$. Again, we have
\[
    \begin{split}
\rho&=\frac{4c_{1}(2n-1)M^{2}}{2m^{3}n^{2}}\eta_{1}(n-1)+\frac{1}{(n-1)m}2\eta_{1}(n-1)\\
&=\frac{M^{2}}{m^{3}}\left(\frac{4c_{1}(2n-1)(n-1)}{2n^{2}}+\frac{2m^{2}}{M^{2}}\right)\eta_{1}\\
&\leq\frac{c_{3}M^{2}\eta_{1}}{m^{3}},
\end{split}
    \]
where $c_{3}$ represents a constant that is independent of $M,m$ and $\eta_{1}$. Combining with (\ref{eq3}), we get
 \[
    \begin{split}
\rho r&\leq\frac{c_{3}M^{2}\eta_{1}}{m^{3}}\left(\frac{c_{2}M^{2}}{nm^{3}}(\sqrt{n\log n}+\kappa\sqrt{\log n})\right)\\
&\leq O\left(\frac{M^{4}\eta_{1}}{nm^{6}}(\sqrt{n\log n}+\kappa\sqrt{\log n})\right)=o(1).
\end{split}
    \]
 Thus, all conditions in Theorem 7 of \cite{Yan:Leng:Zhu:2016} are satisfied. By Lemma \ref{lemma3}, inequality (\ref{equ1}) holds with probability converging to one, thus ensuring the establishment of
(\ref{eq4}).

\subsection{ Proof of Theorem  \ref{theorem2}}

\par \noindent{{\bf Proof of Lemma \ref{lemma4}.}} There are two cases to consider. \\
(i)
By \cite{Yan:2021}, we have
\begin{equation}\label{eq:pro3:b}
| \sum_{i=1}^n \gamma_i | = O_p( \kappa(n\log n)^{1/2} ),~~~~| \sum_{j=1}^n \zeta_j | = O_p(\kappa(n\log n)^{1/2} ).
\end{equation}
Since $\tilde{g}_i -  g_i = \gamma_i$ for $i=1, \ldots, n$, and $\tilde{ g}_{n+j} - g_{n+j} = \zeta_j$ for $j=1, \ldots, n-1$, we have
\begin{eqnarray*}
[S( \bs{\tilde{g}} - \E \bs g)]_i & = & [S( \bs{\tilde{g}} -\bs g )]_i+[S(\bs g - \E\bs  g )]_i \\
& = & (-1)^{1(i,j>n)}\frac{ \sum_{i=1}^n \gamma_i - \sum_{j=1}^{n-1} \zeta_j }{ u_{2n,2n}}+[S(\bs g - \E \bs g )]_i  \\
& = & [S(\bs g - \E\bs g )]_i + O_p( \frac{\kappa(\log n)^{1/2}}{n^{1/2}m_{u}} ).
\end{eqnarray*}
So if $\frac{\kappa(\log n)^{1/2}}{m_{u}}=o(1)$, then we have
\begin{eqnarray*}
[S( \bs{\tilde{g}} - \E \bs g)]_i&=&[S(\bs g - \E \bs g )]_i + o_p( n^{-1/2} ).
\end{eqnarray*}
Consequently, the first part of Lemma \ref{lemma4} immediately follows Proposition 1 of \cite{fan:2021}.\\
(ii) $s_n/v_{2n,2n}^{1/2}$ converges to constant $c$.
Let $\tilde{e}=\sum_{i=1}^n \gamma_i - \sum_{j=1}^{n-1} \zeta_j$ and $\tilde{a}_{i,j} = a_{i,j} - \E a_{i,j}$. Denote
\begin{eqnarray*}
H:=\left(
\begin{array}{c}
\frac{ \tilde{g}_1 - \E g_1 }{ u_{1,1}^{1/2} } \\
\vdots
\\
\frac{ \tilde{g}_k - \E g_k }{ u_{k,k}^{1/2} } \\
\frac{\tilde{g}_{2n}- \E g_{2n} }{ u_{2n,2n}^{1/2} } \\
\frac{ \tilde{e}  }{s_n}
\end{array}
\right)
&=& \left(
\begin{array}{c}
\frac{ \sum_{j=1}^k \tilde{a}_{1,j}  }{ u_{1,1}^{1/2} } \\
\vdots
\\
\frac{ \sum_{j=1}^k \tilde{a}_{k,j}  }{ u_{k,k}^{1/2} } \\
\frac{ \sum_{i=1}^k \tilde{a}_{i,n} }{ u_{2n,2n}^{1/2} } \\
0
\end{array}
\right)
+ \left(
\begin{array}{c}
\frac{ \sum_{j=k+1}^n \tilde{a}_{1,j}  }{ u_{1,1}^{1/2} } \\
\vdots
\\
\frac{ \sum_{j=k+1}^n \tilde{a}_{k,j}  }{ u_{k,k}^{1/2} } \\
\frac{ \sum_{i=k+1}^n \tilde{a}_{i,n} }{ u_{2n,2n}^{1/2} } \\
\frac{ \tilde{e}  }{s_n}
\end{array}
\right) \\
& := & I_1 + I_2.
\end{eqnarray*}
Since $|a_{i,j}| \le 1$ and $u_{i,i}\to\infty$ as $n\to\infty$, $|\sum_{j=1}^k \tilde{a}_{i,j}|/u_{i,i}=o(1)$ for $i=1, \ldots, k$ with  fixed $k$.
So $I_1=o(1)$.

Next, we will consider $I_2$. Recall that $s_n^2=\mathrm{Var}(\tilde{e})$. According to the large sample theory, $(\tilde{e} - \E \tilde{e})/s_n$ is asymptotically standard normal distribution
if $s_n\to\infty$. Furthermore, based on the central limit theorem for the bounded case presented in \cite{Loeve:1977} (page 289), $ \sum_{j=k+1}^n \tilde{a}_{i,j} / u_{i,i}^{1/2} $
is asymptotically standard normal distribution for any fixed $i$ if $M/m_{u}=o(n)$.
Since $\tilde{a}_{i,j}$'s ($1\le i \le k$, $j=k+1, \ldots, n$), $\tilde{a}_{i,n}$'s and
$\tilde{e}$ are mutually independent, $I_2$ is asymptotically
a $k+2$-dimensional standardized normal distribution with covariance matrix $I_{k+2}$, where $I_{k+2}$ denotes the $(k+2)\times (k+2)$ dimensional identity  matrix.
Let
\[
C = \left ( \begin{array}{cccccc}
\frac{\sqrt{u_{1,1}}}{ {v_{1,1}}}, & 0,  & \ldots, & 0, & \frac{\sqrt{u_{2n,2n}}}{ { v_{2n,2n} }}, & \frac{s_n}{v_{2n,2n}} \\
0, &\frac{\sqrt{u_{2,2}}}{ {v_{2,2}}},& \ldots, & 0, &\frac{\sqrt{u_{2n,2n}}}{ { v_{2n,2n} }}, & \frac{s_n}{v_{2n,2n}} \\
  && \ldots &&& \\
0, &0, & \ldots, & \frac{\sqrt{u_{k,k}}}{ {v_{k,k}}}, & \frac{\sqrt{u_{2n,2n}}}{ { v_{2n,2n} }}, & \frac{s_n}{v_{2n,2n}} \\
\end{array}
\right ).
\]
Then
\[
[S(\bs{\tilde{g}}-\E \bs g )]_{i=1, \ldots, k} = C H.
\]
Since $s_n^2/v_{2n,2n} \to c^2$ for some constant $c$, all positive entries of $C$ are in the same order $n^{1/2}$.
So $CH$ is asymptotically $k$-dimensional multivariate normal distribution with mean $(\overbrace{0, \ldots, 0}^{k})$ and covariance matrix
\[
\mathrm{diag}( \frac{u_{1,1}}{v_{1,1}^{2}}, \ldots, \frac{u_{k,k}}{v_{k,k}^{2}})+ (\frac{u_{2n,2n}}{v_{2n,2n}^{2}} + \frac{s_n^2}{v_{2n,2n}^2}) \mathbf{1}_k \mathbf{1}_k^\top,
\]
where $\mathbf{1}_k$ denotes a $k$-dimensional column vector with all elements equal to $1$.
Before proving Theorem \ref{theorem2}, we first establish a lemma.

\begin{lemma}\label{Lemma9}
If \eqref{eq5} holds, then for any $i$,
\begin{equation}\label{5.2}
\bs{\widehat{\theta}}_{i}- \bs{\theta}^{*}_{i}=[V^{-1}(\bs{\tilde{g}}-\mathbb{E}(\bs g))]_{i}+o_p(n^{-1/2}).
\end{equation}
\end{lemma}
\begin{proof}
Let $\hat{r}_{i,j}=\hat{\alpha}_{i}+\hat{\beta}_{j}-\alpha^{*}_{i}-\beta^{*}_{j}$ and assume

\[
  \hat{\rho}_{n} :=\max_{i\neq j}|\hat{r}_{i,j}|=O_{p}\left(\frac{c_{2}M^{2}}{nm^{3}}(\sqrt{n\log n}+\kappa\sqrt{\log n})\right).
  \]
For any $1\leq i\neq j\leq n,$ by the Taylor's expansion, we get
\[
    \begin{split}
\mu(\hat{\alpha}_{i}+\hat{\beta}_{j})-\mu(\alpha^{*}_{i}+\beta^{*}_{j})=\mu'(\alpha^{*}_{i}+\beta^{*}_{j})\hat{r}_{i,j}+h_{i,j},
    \end{split}
    \]
where $h_{i,j}=\cfrac{1}{2}\mu''(\hat{\theta}_{i,j})\hat{r}_{i,j}^{2}$ and $\hat{\theta}_{i,j}=t_{i,j}(\alpha^{*}_{i}+\beta^{*}_{j})+(1-t_{i,j})(\hat{\alpha}_{i}+\hat{\beta}_{j}), t_{i,j}\in (0,1).$ By (\ref{eqqq1}), we have
\[
   \bs{\tilde{g}}-\E(\bs{g})=V({\widehat{\bs\theta}}-\bs{\theta}^{*})+\textbf{h}.
 \]
Equivalently,
    \begin{equation}\label{eq8}
 {\widehat{\bs\theta}}-\bs{\theta}^{*}=V^{-1}( {\bs{\tilde{g}}}- \E(\bs{g}))+V^{-1}\textbf{h}.
    \end{equation}
where $\textbf{h}=(h_{1},\ldots,h_{2n-1})^{\top}$ and
\[
    \begin{split}
 &h_{i}=\sum_{k=1,k\neq i}^{n}h_{i,k}, i=1,\ldots,n,\\
 &h_{n+i}=\sum_{k=1,k\neq i}^{n}h_{k,i}, i=1,\ldots,n-1.
    \end{split}
    \]
By (\ref{eqq4}), we get
\[
    \begin{split}
 & |h_{i,j}|=\Big|\cfrac{1}{2}\mu''(\hat{\theta}_{i,j})\hat{r}_{i,j}^{2}\Big|
\leq \frac{1}{2}\eta_{1}\hat{r}^{2}_{i,j}\leq \frac{1}{2}\eta_{1}\hat{\rho}^{2}_{n},\\
 & |h_{i}|\leq\sum_{j\neq i}\Big|h_{i,j}\Big|\leq \frac{1}{2}(n-1)\eta_{1}\hat{\rho}^{2}_{n}.
    \end{split}
    \]
Note that $(S\textbf{h})_{i}=\cfrac{h_{i}}{v_{i,i}}+(-1)^{1{(i>n)}}\cfrac{h_{2n}}{v_{2n,2n}}$, $h_{2n}:=\sum_{i=1}^{n}h_{i}-\sum_{i=n+1}^{2n-1}h_{i}=\sum_{i=1,i\neq k}^{n}h_{i,n}\leq \cfrac{n-1}{2}\eta_{1}\hat{\rho}^{2}_{n}$ and $(V^{-1}\textbf{h})_{i}=(S\textbf{h})_{i}+(W\textbf{h})_{i}.$
By Proposition 1 of \cite{Yan:Leng:Zhu:2016}, we get
\[
\begin{split}
|(V^{-1}\textbf{h})_{i}|&\leq |(S\textbf{h})_{i}|+|(W\textbf{h})_{i}|\\
&\leq \cfrac{|h_{i}|}{v_{i,i}}+\cfrac{|h_{2n}|}{v_{2n,2n}}+\|W\|_{\max}\times \left[(2n-1)\max_{i}|h_{i}|\right]\\
&\leq O\left(\cfrac{\eta_{1}\hat{\rho}^{2}_{n}}{m}\right)+O\left(\frac{M^{2}}{m^{3}(n-1)^{2}}\times\frac{1}{2}(2n-1)(n-1)\eta_{1}\hat{\rho}_{n}^{2}\right)\\
&\leq O\left(\frac{m^{2}+M^{2}}{2m^{3}}\eta_{1}\hat{\rho}_{n}^{2}\right)\\
&=O\left(\frac{M^{4}(m^{2}+M^{2})}{2m^{9}n^{2}}\eta_{1}(n+2\kappa n^{1/2}+\kappa^{2})\log n\right)\\
&\leq O\left((n+2\kappa n^{1/2}+\kappa^{2})\frac{M^{6}\eta_{1}\log n}{m^{9}n^{2}}\right) .
  \end{split}
    \]
Again, by (\ref{eq5}), we obtain
\begin{equation}\label{eqA4}
|(V^{-1}\textbf{h})_{i}|\leq |(S\textbf{h})_{i}|+|(W\textbf{h})_{i}|=o(n^{-1/2}).
    \end{equation}
\end{proof}

\noindent{{\bf Proof of Theorem  \ref{theorem2}.}
In view of Lemmas \ref{lemma4} and \ref{Lemma9}, the left arguments are very similar those
 Theorem 2 in \cite{Yan:2021}. Thus, we
omit it.

\end{document}